\documentclass[a4paper, 11pt]{article} % Font size (can be 10pt, 11pt or 12pt) and paper size (remove a4paper for US letter paper)
\pdfoutput=1

\usepackage[protrusion=true,expansion=true]{microtype} % Better typography
\usepackage{graphicx} % Required for including pictures
\usepackage{wrapfig} % Allows in-line images

\usepackage{mathpazo} % Use the Palatino font
\usepackage[T1]{fontenc} % Required for accented characters
\makeatletter
\renewcommand\@biblabel[1]{\textbf{#1.}} % Change the square brackets for each bibliography item from '[1]' to '1.'
\renewcommand{\@listI}{\itemsep=0pt} % Reduce the space between items in the itemize and enumerate environments and the bibliography

\renewcommand{\maketitle}{ % Customize the title - do not edit title and author name here, see the TITLE block below
\begin{center}
{\LARGE\@title} % Increase the font size of the title

\vspace{35pt} % Some vertical space between the title and author name

{\large\@author} % Author name
\\\@date % Date

\vspace{20pt} % Some vertical space between the author block and abstract
\end{center}
}

\title{\textbf{The asymptotically flat scalar-flat Yamabe problem with boundary}}%\\ % Title
\author{\textsc{Stephen\hspace{-1.5mm}\newlength{\Mheight}
\newlength{\cwidth}
\settoheight{\Mheight}{M}\settowidth{\cwidth}{c}M\parbox[b][\Mheight][t]{\cwidth}{c}\hspace{0.15mm}Cormick\footnote{stephen.mccormick@une.edu.au}}
\\{\textit{School of Science and Technology\\University of New England\\Armidale, 2351\\Australia}}} % Institution
\usepackage{geometry,color}

\usepackage[latin1]{inputenc}
\usepackage{amsthm}
\usepackage{amsmath,amsthm,amssymb,subdepth}
\usepackage{amsfonts,mathrsfs,enumerate}
\usepackage{setspace}
\usepackage{url}
\newtheorem{theorem}{Theorem}[section]

\newtheorem{corollary}[theorem]{Corollary}

\newtheorem{lemma}[theorem]{Lemma}

\newtheorem{prop}[theorem]{Proposition}
\newtheorem{remark}[theorem]{Remark}

\theoremstyle{plain}

\newtheorem*{prop*}{Proposition}
\newtheorem*{theorem*}{Theorem}

\newcommand{\onabla}{\mathring\nabla}

\newcommand{\og}{\mathring g}

\newcommand{\R}{\mathbb{R}}

\newcommand{\W}{\overline{W}}
\newcommand{\bH}{\overline{H}}

\numberwithin{equation}{section}
\DeclareMathOperator{\ran}{ran}

\begin{document}
\maketitle
\begin{abstract}
We consider two cases of the asymptotically flat scalar-flat Yamabe problem on a non-compact manifold with boundary, in dimension $n\geq3$. First, following arguments of Cantor and Brill in the compact case, we show that given an asymptotically flat metric $g$, there is a conformally equivalent asymptotically flat scalar-flat metric that agrees with $g$ on the boundary. We then replace the metric boundary condition with a condition on the mean curvature: Given a function $f$ on the boundary that is not too large, we show that there is an asymptotically flat scalar-flat metric, conformally equivalent to $g$ whose boundary mean curvature is given by $f$. The latter case involves solving an elliptic PDE with critical exponent using the method of sub- and supersolutions. Both results require the usual assumption that the Sobolev quotient is positive.
\end{abstract}

\section{Introduction}
The Yamabe problem, asked  in its original form, for which compact Riemannian manifolds $(\mathcal M,g)$ does there exist a metric conformal to $g$ with constant scalar curvature \cite{trudinger1968remarks,yamabe1960deformation}. Since the problem was first posed, it has branched out into several related problems; prescribed scalar curvature \cite{kazdanwarner}, the case where $\mathcal M$ has a boundary \cite{escobar1992yamabe}, or is non-compact \cite{CantorBrill1981}, non-compact with boundary \cite{schwartz2006zero}, constant mean curvature boundary \cite{escobar1992conformal}, prescription of boundary mean curvature \cite{EscobarprescribedH}, and others (see \cite{brendle2010recent,yamabe} and references therein).

The non-compact case seems to have only recently been receiving an amount of attention due to its relation to the Einstein constraint equations in general relativity. The Yamabe problem on non-compact manifolds is intimately connected to the conformal method of solving the constraint equations (see, for example, \cite{BaIsConstraints2004,maxwell2005solutions}); in the time-symmetric case, the constraint equations simply reduce to the prescribed scalar curvature equation. Indeed, we are also motivated to study the Yamabe problem on an asymptotically flat manifold with boundary due to this connection to the constraint equations; however, rather than seeking to parametrise the space of solutions, we are motivated by the problem of quasilocal mass (see \cite{QLMreview} for an in-depth review of the problem). Bartnik's quasilocal mass gives the mass of a bounded region as the infimum of the ADM mass, over a space of admissible asymptotically flat extensions to the given region \cite{qlm}. It is expected that this infimum will be achieved by a static metric that is continuous across the boundary and induces the same mean curvature of the boundary, although it is not expected that this solution will be differentiable across the boundary. For this reason, one is interested in finding scalar-flat asymptotically flat solutions with boundary, such that the metric and mean curvature are both fixed on the boundary. In the literature, this boundary data is often called Bartnik data, or Bartnik's geometric boundary data.

Unfortunately, the Yamabe problem with both of these boundary conditions imposed is overdetermined, so we cannot expect to use this method to find vacuum solutions with prescribed Bartnik boundary data. However, it does motivate the following two results. We demonstrate the existence of scalar-flat metrics in a fixed conformal class satisfying either of these boundary conditions separately, but not both simultaneously. In the following two Propositions, we assume $(\mathcal M,g)$ is an asymptotically flat manifold with boundary and positive Sobolev quotient.

\begin{prop*}[See Corollary \ref{cor1} for formal statement]
There exists an asymptotically flat scalar-flat metric on $\mathcal M$, conformal to $g$ and agreeing with $g$ on $\partial \mathcal M$.
\end{prop*}

\begin{prop*}[See Corollary \ref{cor2} for formal statement]
There is a positive function $\rho$ on $\partial \mathcal M$, such that for any function $f<\rho$ on $\partial \mathcal M$, there exists an asymptotically flat scalar-flat metric on $\mathcal M$, conformal to $g$, such that $f$ is the mean curvature of the boundary with respect to the outer unit normal.
\end{prop*}

We prove the first of these results following an argument of Cantor and Brill \cite{CantorBrill1981}, where they prove the analogous result in the compact case. The second of these results uses the method of sub- and supersolutions, and relies heavily on excellent previous work of Maxwell where asymptotically flat solutions to the Einstein constraint equations with apparent horizon boundary conditions are constructed \cite{maxwell2005solutions}. This second result is also related to previous work by Schwartz \cite{schwartz2006zero}, which considered the scalar-flat Yamabe problem outside of a ball, however no assumptions about asymptotic flatness are made in this case considered by Schwartz.

\section{The setup}
Throughout, we use the framework of weighted Sobolev spaces to control the asymptotics. In this Section we recall some properties of the Laplace-Beltrami operator between these weighted spaces on an asymptotically flat manifold with boundary. Let $\mathcal M$ be a smooth manifold with boundary such that $\mathcal{M}$ minus a compact set $K$ containing the boundary ${\partial\mathcal M}$, is diffeomorphic to $\R^n$ minus a closed ball. Denote this diffeomorphism by $\phi:\mathcal M\setminus K\rightarrow\R^n\setminus \overline{B_1(0)}$. On $\mathcal{M}$, we fix some smooth background metric $\og$ that agrees with $\phi^*(\delta)$, the pullback of the Euclidean metric, on $\mathcal M\setminus K$. We also fix a smooth function $r(x)$ on $\mathcal M$ such that $r(x)\geq1$ and $r(x)=|\phi(x)|$ on $\mathcal M\setminus K$. We now recall the weighted Lebesgue and Sobolev norms:
\begin{align}
\left\|u\right\|_{p,\delta}&=
\left\{
\begin{array}{ll}
\left(\int_{\mathcal M}\left| u\right|^p r^{-\delta p-n}d\mu_0\right)^{1/p},& p<\infty,\\
\text{ess sup}(r^{-\delta}|u|), & p=\infty,
\end{array}\label{weighted1}
\right.
\\
\left\|u\right\|_{k,p,\delta}&=\sum_{j=0}^k\|\onabla^j u\|_{p,\delta-j},\label{weighted2}
\end{align}
where $\circ$ refers to quantities defined by $\og$. The spaces $L^p_\delta(\mathcal M)$ and $W^{k,p}_\delta(\mathcal M)$ are defined as the completion of the smooth functions with bounded support on $\mathcal M$ with respect to these norms, respectively. We follow the convention of \cite{AF}, where $\delta$ explicitly indicates the asymptotics; that is, if $u\in W^{k,p}_{\delta}$ then $u=o(r^\delta)$. Denote by $\overline{W}^{k,p}_\delta(\mathcal M)$, the completion of the compactly supported functions on $\mathcal M\setminus{\partial\mathcal M}$ with respect to the $W^{k,p}_\delta$ norm. That is, $\overline{W}^{k,p}_\delta(\mathcal M)$ is a space of functions that vanish on the the boundary in the trace sense, along with their first $k-1$ derivatives. We will generally omit reference to $\mathcal{M}$, and simply write $W^{k,p}_\delta$ for the sake of presentation. We say that $(\mathcal M,g)$ is an asymptotically flat manifold if $(g-\og)\in W^{k,2}_{5/2-n}$ with $k>n/2$, which ensures $g$ is H\"older continuous via the Sobolev-Morrey embedding. It is well-known that the usual Sobolev-type inequalities have weighted analogues, which we use throughout -- the reader is referred to Theorem 1.2 of \cite{AF} for an explicit statement of many of these inequalities.

It is also well-known that the Laplace-Beltrami operator is an isomorphism between these weighted spaces in the case where $\mathcal M$ has no boundary. This is also true in the case considered here, when Dirichlet boundary conditions are imposed. In particular, we make use of the following results (cf. \cite{McCormickExtensionsManifold}).

\begin{lemma}\label{lemineq}
Let $M$ be an asymptotically flat manifold with a compact interior boundary. For any $\delta\in(2-n,0)$ and $u\in W^{2,2}_\delta\cap\overline{W}^{1,2}_\delta$, we have
\begin{equation}
\|u\|_{2,2,\delta}\leq C\|\Delta_g u\|_{2,\delta-2}.\label{coercivelaplace}
\end{equation}
Furthermore, for any $\delta\in\R$ and $u\in W^{2,2}_\delta$, we have
\begin{equation}\label{nonfredholm}
\|u\|_{2,2,\delta}\leq C\left(\|\Delta_g u\|_{2,\delta-2}+\|u\|_{2,\delta}\right).
\end{equation}
\end{lemma}
\begin{prop}\label{propisolaplace}
For any $\delta\in(2-n,0)$, $\Delta_{g}:\W^{1,2}_\delta \cap W^{2,2}_\delta \rightarrow L^2_{\delta-2}$ is an isomorphism.
\end{prop}
We now use Lemma \ref{lemineq} to establish the following straightforward corollary.
\begin{corollary}\label{corfR}
Assume $\delta\in (2-n,0)$ and $f\in L^n_{-2-\epsilon}$, with $\epsilon>0$, then the operator ${L=\Delta_g+f:\W^{1,2}_\delta \cap W^{2,2}_\delta \rightarrow L^2_{\delta-2}}$ has finite dimensional kernel and closed range.
\end{corollary}
\begin{proof}
Making use of (\ref{coercivelaplace}), and the weighted H\"older, Sobolev and interpolation inequalities, we have
\begin{align}
\|u\|_{2,2,\delta}&\leq C\|Lu\|_{2,\delta-2}+\|fu\|_{2,\delta-2}\nonumber\\
&\leq C(\|Lu\|_{2,\delta-2}+\|f\|_{n,-2-\epsilon}\|u\|_{2n/(n-2),\delta+\epsilon})\nonumber\\
&\leq C(\|Lu\|_{2,\delta-2}+\|u\|_{1,2,\delta+\epsilon})\nonumber\\
&\leq C(\|Lu\|_{2,\delta-2}+\|u\|_{2,\delta+\epsilon})+\frac12\|u\|_{2,2,\delta+\epsilon}\nonumber\\
&\leq C(\|Lu\|_{2,\delta-2}+\|u\|_{2,\delta+\epsilon}).\label{LRscalebroken}
\end{align}
It then follows by a standard argument (see, for example, the proof of Theorem 1.10 in \cite{AF}) that $L$ has finite dimensional kernel and closed range. Let $u_i$ be a sequence in $\ker(L)$ satisfying $\|u\|_{2,2,\delta}\leq1$; that is, a sequence in the closed unit ball in $\ker(L)$. By the weighted Rellich compactness theorem, and passing to a subsequence, we have $u_i\rightarrow u$ in $L^2_{\delta+\epsilon}$. Then by the estimate above, we conclude $u_i$ also converges in $W^{2,2}_\delta$, and it follows that $\ker(L)$ is finite dimensional. It follows that the domain splits as $\W^{1,2}_\delta \cap W^{2,2}_\delta=\ker(L)\oplus K$, where $K$ is some closed complementary subspace. To prove that $\ran(L)$ is closed, we require the following estimate for all $u\in K$,
\begin{equation}\label{Kestimate}
\|u\|_{2,2,\delta}\leq C\|Lu\|_{2,\delta-2},
\end{equation}
which we prove by contradiction. If (\ref{Kestimate}) were not true, there must exist a sequence $u_i$ satisfying $\|u_i\|_{2,2,\delta}=1$ while $\|Lu_i\|_{2,\delta-2}\rightarrow0$. But again, (\ref{LRscalebroken}) and the weighted Rellich compactness theorem imply $u_i\rightarrow u$ in $W^{2,2}_\delta$, passing to a subsequence if necessary. Since $K$ is closed, we have $u\in K$, and therefore $u$ is a nonzero element of both $K$ and $\ker(L)$, which is a contradiction. It follows that (\ref{Kestimate}) holds, and therefore for any Cauchy sequence in the range $v_i=Lu_i$, $u_i$ is also Cauchy and therefore converges in $W^{2,2}_\delta$. By continuity, $v_i$ converges and therefore $\ran(L)$ is closed.

\end{proof}
Note that the regularity of $f$ is almost certainly not optimal, however it suffices for our purposes here.

%%%%%%%%%%%%%%%%%%%%%%%%%%%%%%%%%%%%%%%%%%%%%%%%%%%%%%%%%%%%%%%%%
\section{Dirichlet boundary conditions}\label{Syamabe}
We are now ready to discuss the asymptotically flat scalar-flat Yamabe problem, with Dirichlet boundary conditions on the metric. Given some fixed metric $g$, consider conformally related metrics of the form $\tilde{g}=\phi^{4/(n-2)}g$. The problem of prescribing scalar curvature $\tilde{R}:=R(\tilde g)=f$ is equivalent to solving
\begin{equation}\label{yamabe}
\left\{\begin{matrix}
\frac{4(n-1)}{n-2}\Delta_g\phi-R\phi+\phi^{(n+2)/(n-2)}f=0\\
\phi>0
\end{matrix}\right. .
\end{equation}
The Yamabe problem is historically concerned with the case where $f$ is constant, and the case of particular interest here is the case where $\tilde R=f\equiv0$. Recall, asymptotically flat scalar-flat 3-manifolds correspond to initial data for the vacuum Einstein equations. Note that this case is significantly simpler than the general case as it removes the problem of the critical Sobolev exponent. The system under consideration in this Section is
\begin{equation}\label{yamabe2}
\left\{\begin{matrix}
\frac{4(n-1)}{n-2}\Delta_g\phi-R\phi=0&\text{ on }\mathcal M\\
\phi>0&\\
\phi-1=o(r^{5/2-n})&\\
\phi\equiv1&\text{ on } \partial\mathcal M
\end{matrix}\right. 
\end{equation}
where $g$ is assumed to be asymptotically flat. The argument we use here follows that of Cantor and Brill \cite{CantorBrill1981} in the case where no boundary is present. It has been noted by Maxwell \cite{maxwell2005solutions} that the Cantor-Brill proof has a minor error, so we are careful avoid this error here.

Recall now, the Sobolev quotient, given by
\begin{equation}
Q(\mathcal M,g)=\inf_{f\in C^\infty_c}\frac{\int_{\mathcal M} |\nabla f|^2+\frac{n-2}{4(n-1)}Rf^2 d\mu}{\|f\|^2_{L^{2n/(n-2)}}},
\end{equation}
which is conformally invariant, and intimately connected with the solvability of the Yamabe problem.

\begin{lemma}
Let $(\mathcal M,g)$ be an asymptotically flat manifold with $Q(\mathcal M,g)>0$. For each $\lambda\in[0,1]$, the operator $A_\lambda=\frac{4(n-1)}{n-2}\Delta_g-\lambda R$ is injective on $u\in H^2_{1-\frac n2}\cap \bH^1_{1-\frac n2}$.
\end{lemma}
\begin{proof}

Suppose $u\in H^2_{1-\frac n2}\cap \bH^1_{1-\frac n2}$ satisfies $A_\lambda u=0$, and let $u_m\in C^\infty_c$ be a sequence converging to $u$ in $\bH^1_{1-\frac n2}$. We have
$$\int_{\mathcal M} (\frac{4(n-1)}{n-2}u_m\Delta_g u-\lambda R u_m u) d\mu=0.$$
Integrating by parts,
$$-\frac{4(n-1)}{n-2}\int_{\mathcal M} \nabla_i u_m\nabla^i u\, d\mu=\lambda\int_{\mathcal M}Ru_m u\, d\mu.$$
From the weighted H\"older and Sobolev inequalities, we have
\[|\int_{\mathcal M} \nabla_i v\nabla^i u\, d\mu|\leq C\|v\|_{1,2,1-n/2}\|u\|_{1,2,1-n/2}\]
and
\begin{align*}
 |\int_{\mathcal M}Rv u\, d\mu&|\leq C\|R\|_{n/2,-2}\|v\|_{2n/(n-2),1-n/2}\|u\|_{2n/(n-2),1-n/2}\\
&\leq C\|R\|_{n/2,-2}\|v\|_{1,2,1-n/2}\|u\|_{1,2,1-n/2},
\end{align*}
noting that $\|R\|_{n/2,-2}$ is finite by assumption of asymptotic flatness. It follows that the maps, $v\mapsto\int_\mathcal{M}\nabla_i(v)\nabla^i(u)\, d\mu$ and $v\mapsto\int_\mathcal{M} Ruv\,d\mu$ are continuous. Passing to the limit we conclude
$$-\frac{4(n-1)}{n-2}\int_{\mathcal M} |\nabla u|^2\, d\mu=\lambda\int_{\mathcal M}Ru^2\, d\mu.$$
We have already that $A_0$ is an isomorphism, so we may assume $\lambda\in(0,1]$. If $u\neq0$, we have
$$\frac{4(n-1)}{\lambda(n-2)}\int_{\mathcal M} |\nabla u|^2\, d\mu=-\int_{\mathcal M}Ru^2\, d\mu<\frac{4(n-1)}{n-2}\int_{\mathcal M} |\nabla u|^2\, d\mu$$
from the assumption that $Q(\mathcal M,g)>0$. This cannot hold for $\lambda \leq1$ so we therefore conclude $u\equiv0$ and therefore $A_\lambda$ is an injection.
\end{proof}

We next make use of the following well-known lemma (cf. \cite{cantor2,CantorBrill1981}).
\begin{lemma}\label{leminjfam}
Let $E,F$ be Banach spaces and suppose for $\lambda\in[0,1]$, $L_\lambda:E\rightarrow F$ is a continuous family of bounded linear operators. If $L_0$ is an isomorphism and each $L_\lambda$ is an injection with closed range, then each $L_\lambda$ is in fact an isomorphism.
\end{lemma}
From this, we establish:

\begin{prop}\label{propyamabe1}
Assume $(g-\og)\in H^k_{5/2-n}$ for some $k> n+2$, and $Q(\mathcal M,g)>0$, then there exists $\phi$ with $(\phi-1)\in H^2_{5/2-n}\cap\bH^1_{5/2-n}$ satisfying  (\ref{yamabe2}).
\end{prop}
\begin{proof}
By assumption, $R\in L^n_{1/2-n}$, so $A_{\lambda}$ is an injection and by Corollary \ref{corfR}, it has closed range; that is, by Lemma \ref{leminjfam} we have that $A_1$, in particular, is an isomorphism.

Now let $v=\phi-1$ and note that (\ref{yamabe2}) requires
\begin{equation}\label{yamabemin1}
\frac{4(n-1)}{n-2}\Delta_g v-Rv=R.
\end{equation}
Since $A_1$ is an isomorphism, we have a unique solution $v\in (H^2_{5/2-n}\cap\overline{H}^1_{5/2-n})$. Furthermore, standard elliptic theory implies that $v$ is $C^{1,\alpha}$.

It remains to be shown that $\phi$ is positive, which follows identically from the Cantor-Brill argument, and essentially is as follows. For each $\lambda\in[0,1]$, there exists a unique $\phi_\lambda$ with $(\phi_\lambda-1)\in(H^2_{5/2-n}\cap \bH^1_{5/2-n}\cap C^{1,\alpha}_{loc})$ such that $A_\lambda\phi_\lambda=0$. Furthermore, $\phi_\lambda$ depends continuously on $\lambda$ in the $C^0$ topology. Since $\phi_0\equiv1$, if there is any $\phi_\lambda$ that is not strictly positive, then there must be a $\phi_{\lambda_0}\geq0$ and $x_0\in\mathcal{M}\setminus{\partial\mathcal M}$ such that $\phi_{\lambda_0}(x_0)=0$, which then implies $\nabla\phi_{\lambda_0}(x_0)=0$. However, a classical result (Theorem A of \cite{aleksandrov1962uniqueness}) says that if $u_1\leq u_2$ satisfy $A_{\lambda_0}u_1=0=A_{\lambda_0}u_2$ on some region $G$, and there is a point $x_0\in G$ such that $u_1(x_0)=u_2(x_0)$ and $\nabla u_1(x_0)=\nabla u_2(x_0)$, then $u_1=u_2$ on $G$ (cf. \cite{cantoryorkdata1979}, Theorem 1.7 and the proof of Theorem 2.3). That is, $\phi_{\lambda0}\equiv0$, which is a contradiction and it therefore follows that $\phi>0$.
\end{proof}
From which we have the following immediate Corollary.
\begin{corollary}\label{cor1}
Assume $(g-\og)\in H^k_{5/2-n}$ for some $k> n+2$, and $Q(\mathcal M,g)>0$, then there exists an asymptotically flat scalat-flat metric, conformal to $g$, that agrees with $g$ on $\partial \mathcal M$.
\end{corollary}

\begin{remark}
The assumption $k> n+2$ is certainly not sharp here, however this allow us to avoid complications by simply quoting standard elliptic theory.
\end{remark}

\section{Prescribed boundary mean curvature}
In this Section, we give conditions under which a function on $\partial \mathcal M$ can be realised as the mean curvature an asymptotically flat scalar-flat metric in a prescribed conformal class. While the scalar curvature transforms according to (\ref{yamabe}) under a conformal transformation $\tilde{g}=u^{4/(n-2)}g$, the mean curvature of ${\partial\mathcal M}$ transforms according to
\begin{equation}
\tilde{H}=u^{-n/(n-2)}\left(\frac{2}{n-2}\frac{\partial u}{\partial \eta}+Hu\right);
\end{equation}
where $\eta$ is the outward unit normal, pointing away from infinity, and $H$ is the mean curvature with respect to $\eta$. This problem can be simplified greatly by recalling work of Maxwell \cite{maxwell2005solutions} where asymptotically flat solutions to the Einstein constraints with apparent horizon boundary conditions are constructed. In particular, it is shown that the positivity of the Yamabe constant is a sufficient and necessary condition to ensure the existence of a scalar-flat, asymptotically flat metric with minimal surface boundary in a given conformal class. That is, without loss of generality we can assume both $R$ and $H$ vanish. The problem of finding a scalar-flat asymptotically flat metric conformal to $g$, with boundary mean curvature given by some $f$, then reduces to the following:

\begin{equation}\label{mainproblem}
\left\{\begin{matrix}
\Delta_gu=0&\text{ on }\mathcal M\\
\frac{\partial u}{\partial \eta}-\frac{n-2}{n}fu^{n/(n-2)}=0&\text{ on }\partial \mathcal M\\
u>0&
\end{matrix}\right. .
\end{equation}

We absorb the constant $\frac{n-2}n$ into $f$, and consider the more general problem

\begin{equation}\label{generalproblem}
\left\{\begin{matrix}
\Delta_gu=0&\text{ on }\mathcal M\\
\frac{\partial u}{\partial \eta}-fu^{\beta}=0&\text{ on }\partial \mathcal M\\
u>0
\end{matrix}\right. ,
\end{equation}
for any $\beta\in\mathbb{R}$.

The method of sub- and supersolutions has been particularly fruitful in considering the Yamabe problem on non-compact domains \cite{maxwell2005solutions,schwartz2006zero}, and it once again finds use here. The particular theorem regarding sub- and supersolutions that we use here is again due to Maxwell; we state a special case of this result below, which suffices for our purposes.

\begin{prop}[Proposition 2 of \cite{maxwell2005solutions}]
Suppose $g\in W^{k,p}_\rho$ with $k\geq 2$, $p>n/k$ and $\rho<0$. Further suppose that $f\in W^{k-1-\frac{1}{p},p}({\partial\mathcal M})$. Then if $u_-,u_+\in W^{k,p}_\delta$, for $\delta\in(2-n,0)$, are a subsolution and a supersolution of (\ref{generalproblem}) respectively, there is a solution $u$ to (\ref{generalproblem}) satsifying $u_-\leq u\leq u_+$.
\end{prop}
As usual, a subsolution is taken to be $u_-$ satisfying
\begin{equation*}
\left\{\begin{matrix}
\Delta_gu\geq0&\text{ on }\mathcal M\\
\frac{\partial u}{\partial \eta}-fu^{\beta}\leq0&\text{ on }\partial\mathcal M
\end{matrix}\right. ,
\end{equation*}
while supersolution refers to the case where the inequalities are reversed.

We establish the following.
\begin{prop}\label{propyamabe2}
Let $(\mathcal M,g)$ be an asymptotically flat manifold, where $g\in H^k_{5/2-n}$ with $k>n$, and let $f\in H^{k-\frac32}({\partial\mathcal M})$ satisfy $f\leq0$. Then the problem (\ref{generalproblem}) has a solution. Furthermore, if $\beta>1$ then there exists a positive function, $\rho$, on $\partial \mathcal M$ such that the problem  (\ref{generalproblem}) has a solution for any $f<\rho$.
\end{prop}
\begin{proof}
It follows from Proposition \ref{propisolaplace} that there exists $v$ satisfying $\Delta_gv=0$, that is identically equal to $1$ on $\partial\mathcal M$, and $v\in H^2_\delta$ for all $\delta\in(2-n,0)$. Standard elliptic theory again implies $v$ is $C^{1,\alpha}$ and furthermore, note that we have $v=O(r^{2-n})$. By the maximum principle and the Hopf lemma, $0<v\leq1$ on $\mathcal{M}$ and $0<\frac{\partial v}{\partial\eta}$ on $\partial \mathcal M$. We now let $u_-=1-v+\alpha v$ for some $\alpha>0$, noting that $\Delta_g u_-=0$ and we have
\begin{equation}\label{eqsubsuper}
\frac{\partial u_-}{\partial \eta}-fu_-^\beta=-\frac{\partial v}{\partial \eta}+\alpha\frac{\partial v}{\partial\eta}-f\alpha^\beta=\frac{\partial v}{\partial \eta}(\alpha-1)-f\alpha^\beta
\end{equation}
on $\partial \mathcal M$. Since $\frac{\partial v}{\partial \eta}>0$, we have $\frac{\partial u_-}{\partial \eta}-fu_-^\beta<0$ on $\partial \mathcal M$ for sufficiently small $\alpha$. Furthermore, we have $u_-\rightarrow 1$ at infinity, and the maximum principle gives $\alpha\leq u_-<1$. That is, $u_-$ is a subsolution to \ref{generalproblem}.

To find a supersolution, we let $u_+$ be of the same form as $u_-$ above, however we will choose a different $\alpha$. It is clear from (\ref{eqsubsuper}) that we have a supersolution if there exists an $\alpha$ satisfying $f<\frac{\partial v}{\partial\eta}\,(\alpha-1)\alpha^{-\beta}$. This is certainly true if $f<0$ as we may simply chose $\alpha>1$. If we now impose the condition $\beta>1$, we find the maximum occurs when $\alpha=\frac{\beta}{\beta-1}$; that is, for $\beta>1$ a supersolution exists provided that $f<\rho:=\frac{\partial v}{\partial\eta}\,(\frac{1}{\beta-1})^{1-\beta}\beta^{-\beta}$. Once more by the maximum principle, we have $1<u_+\leq\alpha$, and therefore $0<u_-\leq u_+$.
\end{proof}

From this, we have the following immediate corollary.

\begin{corollary}\label{cor2}
Let $(\mathcal M,g)$ satisfy the conditions of Proposition \ref{propyamabe2}, then there exists a positive function $\rho$ on $\partial \mathcal M$ such that for any $f\in H^{k-\frac32}({\partial\mathcal M})$ satisfying $f<\rho$, there is a scalar-flat, asymptotically flat metric $\tilde{g}$, conformal to $g$, whose boundary mean curvature is $f$.

\end{corollary}

\section{Acknowledgements}
The author would like to thank the Institut Henri Poincar\'e, for their hospitality while part of this work was completed, and gratefully acknowledge that this work was supported by a UNE Research Seed Grant.

\bibliographystyle{abbrv}
\bibliography{../../../refsnew}

\end{document}